\documentclass[11pt,a4paper]{article}
\usepackage{amsmath,amssymb,amsthm,amscd,mathrsfs}
\usepackage{indentfirst}
\usepackage[top=25mm, bottom=30mm, left=30mm, right=30mm]{geometry}
\usepackage{color}
\usepackage{comment}

\newtheorem{df}{\bf Definition}[section]
\newtheorem{thm}[df]{\bf Theorem}
\newtheorem{lem}[df]{\bf Lemma}

\newtheorem*{conj}{\bf Haagerup--St\o rmer's Conjecture}

\newtheorem*{claim}{\bf Claim}

\newtheorem{thmA}{\bf Theorem}
\newtheorem{corA}[thmA]{\bf Corollary}

\newcommand{\R}{\mathbb{R}}
\newcommand{\C}{\mathbb{C}}
\newcommand{\Z}{\mathbb{Z}}

\newcommand{\N}{\mathbb{N}}
\newcommand{\B}{\mathbb{B}}

\newcommand{\M}{\mathbb{M}}

\newcommand{\ri}{\mathrm{i}}
\newcommand{\actson}{\curvearrowright}

\newcommand{\Ad}{\operatorname{Ad}}
\newcommand{\id}{\text{\rm id}}

\newcommand{\Aut}{\operatorname{Aut}}

\newcommand{\Tr}{\mathord{\text{\rm Tr}}}
\newcommand{\ovt}{\mathbin{\overline{\otimes}}}

\newcommand{\Int}{\operatorname{Int}}

\title{\bf Haagerup and St\o rmer's conjecture for\\ pointwise inner automorphisms}
\author{Yusuke Isono\thanks{Research Institute for Mathematical Sciences, Kyoto University, 606-8502, Kyoto, Japan \protect \\  E-mail: \texttt{isono@kurims.kyoto-u.ac.jp} \protect \\  
YI is supported by JSPS KAKENHI Grant Number 20K14324.}}
\date{}
\begin{document}
\maketitle

\begin{abstract}
In 1988, Haagerup and St\o rmer conjectured that any pointwise inner automorphism of a type $\rm III_1$ factor is a composition of an inner and a modular automorphism. We study this conjecture and prove that any type $\rm III_1$ factor with trivial bicentralizer indeed satisfies this condition. In particular, this shows that Haagerup and St\o rmer's conjecture holds in full generality if Connes' bicentralizer problem has an affirmative answer. Our proof is based on Popa's intertwining theory and Marrakchi's recent work on relative bicentralizers. 
\end{abstract}

\section{Introduction}

	Let $M$ be a von Neumann algebra and $\Aut(M)$ the set of all automorphisms on $M$. We say that $\theta\in \Aut(M)$ is \textit{pointwise inner} \cite{HS87} if for any positive linear functional $\varphi \in M_*^+$, there is a unitary $u\in \mathcal U(M)$ such that $\theta(\varphi)=u\varphi u^*$, where we used the notation $\theta(\varphi):=\varphi\circ \theta^{-1}$ and $x\varphi y := \varphi(y \, \cdot \, x)$ for $x,y\in M$. 
In other words, $\theta$ is pointwise inner if the induced map $\theta\colon M_\ast^+ \to M_\ast^+$ is inner at a pointwise level. This notion naturally appeared in the study of unitary equivalence relations on state spaces on von Neumann algebras and was intensively studied by Haagerup and St\o rmer \cite{HS87,HS88,HS91}. 

Any inner automorphism is pointwise inner. Indeed, we can choose the same unitary for all elements in $M_\ast^+$. A nontrivial and motivated example comes from \textit{Tomta--Takesaki's modular theory}, that is, any \textit{modular automorphism} $\sigma_t^\varphi$ is pointwise inner, where $\varphi\in M_\ast^+$ is a faithful state and $t\in \R$, see Section \ref{Preliminaries}. 
Hence, any composition $\Ad(u)\circ \sigma^\varphi_t$ of an inner and a modular automorphism is pointwise inner because the pointwise inner property is closed under compositions. Haagerup and St\o rmer conjectured that the opposite phenomena also holds for any type $\rm III_1$ factor \cite{HS88}.

\begin{conj}
	Let $M$ be a type $\rm III_1$ factor with separable predual. Then any pointwise inner automorphism is a composition of an inner and a modular automorphism.
\end{conj}

We note that if $M$ is a factor that is not of type $\rm III_1$, similar results were already obtained \cite{HS87,HS88}, hence the only remaining problem for characterizing pointwise inner automorphisms is the case for type $\rm III_1$ factors. 
We also note that the separability assumption is necessary, see \cite[Section 6]{HS88} and \cite[Remark 4.21]{AH12}.


In the present article, we give a partial but satisfactory answer to this conjecture. To explain it, we prepare terminology. For any inclusion of von Neumann algebras $N\subset M$, we say it is \textit{irreducible} if $N'\cap M \subset N$, and is \textit{with expectation} if there exists a faithful normal conditional expectation from $M$ onto $N$. 
For any faithful state (or semifinite weight) $\varphi\in M_\ast$, we denote by $M_\varphi$ the \textit{centralizer of $\varphi$}, which is the fixed point algebra of $\sigma^\varphi$. 
When $M$ is a type $\rm III_1$ factor with separable predual, the existence of a faithful state $\varphi\in M_\ast$ with $M_\varphi'\cap M=\C$ is equivalent to having \textit{trivial bicentralizer}, see Section \ref{Preliminaries}. In this case, $M_\varphi\subset M$ is an irreducible subfactor of type $\rm II_1$ with expectation, and the existence of such a subfactor is very useful in many contexts. 
\textit{Connes' bicentralizer problem}, which is one of the most important problem for type $\rm III$ factors, asks if any type $\rm III_1$ factor with separable predual has trivial bicentralizer. To the best of our knowledge, all concrete examples of type $\rm III_1$ factors have trivial bicentralizers. Further, if $M$ is any type $\rm III_1$ factor, then the tensor product $M\ovt R_\infty$ has trivial bicentralizer, where $R_\infty$ is the Araki--Woods factor of type $\rm III_1$ \cite{Ma18}.

Now, we introduce our main theorem. For objects in item (3) below, see Section \ref{Preliminaries}.

\begin{thmA}\label{thmA}
	Let $M$ be a type $\rm III_1$ factor with separable predual and assume it has trivial bicentralizer. Fix any faithful state $\varphi\in M_\ast$ such that $M_\varphi'\cap M=\C$. Then for any $\theta\in \Aut(M)$, the following conditions are equivalent.
\begin{enumerate}
	\item The automorphism $\theta$ is pointwise inner.

	\item For any faithful state $\psi\in M_\ast$ with $M_\psi'\cap M =\C$, there exists $u\in \mathcal U(M)$ such that $\theta(x) = uxu^*$ for all $x\in M_\psi$. 

	\item There exists $u\in \mathcal U(M)$ such that $\theta^\omega(x)=uxu^*$ for all $x\in (M^\omega)_{\varphi^\omega}$, where $\omega$ is any fixed free ultrafilter on $\N$.

	\item There exist $u\in \mathcal U(M)$ and $t\in \R$ such that $\theta = \Ad(u)\circ \sigma_t^\varphi$.

\end{enumerate}
\end{thmA}

By the implication (1)$\Rightarrow$(4), we immediately get the following corollary. 
\begin{corA}
	Haagerup--St\o rmer's conjecture holds for any type $\rm III_1$ factor with separable predual that has trivial bicentralizer. In particular, the conjecture holds in full generality if Connes' bicentralizer problem has an affirmative answer.
\end{corA}

We emphasize that Theorem \ref{thmA} is applied to all concrete examples of type $\rm III_1$ factors because they have trivial bicentralizers. Although we do not know the complete answer for Connes' bicentralizer problem, if it is solved affirmatively, then our theorem solves Haagerup--St\o rmer's conjecture. Even if the problem has a negative answer, we can apply the theorem to any type $\rm III_1$ factor up to a tensor product with $R_\infty$. Thus, in any case, Theorem \ref{thmA} covers a large class of type $\rm III_1$ factors.

We explain item (2) in Theorem \ref{thmA}. In all previous works for pointwise inner automorphisms, a specific state (or weight) $\varphi$ on a factor $M$ plays a crucial role. More precisely, the implication (1)$\Rightarrow$(2) in Theorem \ref{thmA} for the specific $\psi  \ (=\varphi)$ is the key step of the proof. 
For example, $\varphi$ is a trace if $M$ is of type $\rm II$. If $M$ is a type $\rm III_\lambda$ $(0\leq \lambda <1)$ factor, then $\varphi$ is a lacunary weight. We note that, in these cases, the implication (2)$\Rightarrow$(4) also follows by the fact that  $\varphi$ is a specific one. 
If $M$ is a type III$_1$ factor, unfortunately, there are no such specific states or weights in general, and this fact is the main difficulty for the conjecture. We do have a dominant weight, but it is not enough for our purpose because $M_\varphi \subset M$ is not with expectation. In \cite{HS91,HI22}, $\varphi$ is assumed to be an almost periodic state, but not all type $\rm III_1$ factors admit such a state.

In Theorem \ref{thmA}, we do not assume $\varphi$ is such a specific state or weight. We nevertheless prove the implication (1)$\Rightarrow$(2) holds. This is the main observation of the article, and we need some techniques from \textit{Popa's intertwining theory} \cite{Po01,Po03}. 
Once we get item (2), then since it is the condition for \textit{sufficiently many $\psi$}, we can use recent results from Marrakchi's work on relative bicentralizers \cite{Ma23} and get item (4). 
Thus, our proof is new even for the Araki--Woods factor $R_\infty$. We note that this is why we can avoid the use of the classification theorem of $\Aut(R_\infty)$ \cite{KST89}, which was necessary in all previous works \cite{HS91,HI22} for the type $\rm III_1$ factor case.

We note that, by a result in \cite{AHHM18}, we can involve item (3) of Theorem \ref{thmA}, 
which is a condition on ultraproducts. This condition is interesting in the sense that $\varphi^\omega$ has more information than $\varphi$. Indeed, item (3) is a condition for a single $\varphi$, while item (2) is one for sufficiently many $\psi$.

\subsection*{Acknowledgments}
	The author would like to thank Yuki Arano,  Toshihiko Masuda, and Reiji Tomatsu for fruitful conversations about automorphisms of von Neumann algebras. He also would like to thank Amine Marrakchi for letting him know about the recent work on relative bicentralizers.

\tableofcontents

\section{Preliminaries}
\label{Preliminaries}

For any von Neumann algebra $M$, the $L^2$-norm with respect to $\varphi\in M_\ast^+$ is denoted by $\| \, \cdot \, \|_{\varphi}$. 
We use the notation  $M_p:=pMp$ for any projection $p\in M \cup M'$.

\subsection*{Tomita--Takesaki theory and Connes cocycles}

	Let $M$ be a von Neumann algebra and $\varphi\in M_\ast^+$ a faithful functional. Throughout the paper, the \textit{modular action} for $\varphi$ is denoted by $\sigma^\varphi\colon \R \actson M$.  The crossed product $M\rtimes_{\sigma^\varphi}\R$ is called the \textit{continuous core} (with respect to $\varphi$) and is denoted by $C_\varphi(M)$. We say that $M$ is a \textit{type $\rm III_1$ factor} if $C_\varphi(M)$ is a factor. The \textit{centralizer algebra} $M_\varphi$ is the fixed point algebra of the modular action $\sigma^\varphi$. 
See \cite{Ta03} for details of these objects.

	Let $\alpha\colon\R \actson M$ be a continuous action and $p \in M$ a nonzero projection. We say that a $\sigma$-strongly continuous map $u\colon \R \to pM$ is a \textit{generalized cocycle for $\alpha$ (with support projection $p$)} if it satisfies
	$$u_{s+t}=u_s \alpha_s(u_t),\quad u_su_s^* = p, \quad u_s^* u_s = \alpha_s(p),\quad \text{for all }s,t\in \R.$$
In this case, by putting $\alpha^u_s(pxp):=u_s\alpha_s(pxp)u_s^*$ for all $x\in M$ and $s\in G$, we have a  continuous $\R$-action on $pMp$.

For $\varphi,\psi\in M_\ast^+$ with $\varphi$ faithful and with $s(\psi)$ the support projection of $\psi$, consider modular actions $\sigma^\varphi$ on $M$ and $\sigma^\psi$ on $M_{s(\psi)}$. 
The \textit{Connes cocycle} $(u_t)_{t\in \R}$ (for $\psi$ with respect to $\varphi$) \cite{Co72} is a generalized cocycle for $\sigma^\varphi$ with support projection $s(\psi)$ such that $(\sigma^\varphi)^{u}\colon \R \actson  M_{s(\psi)}$ coincides with $\sigma^\psi$. We denote it by   $u_t=[D\psi:D\varphi]_t$ for $ t\in \R$. 
See \cite[VIII.3.19-20]{Ta03} for this non-faithful version of the Connes cocycle. 

Let $\theta = \sigma^\varphi_t$ for some $t\in \R$ and we see that it is pointwise inner. Take any $\psi\in M_\ast^+$ and take any faithful $\widetilde{\psi}\in M_\ast^+$ with $s(\psi)\widetilde{\psi}=\widetilde{\psi}s(\psi)=\psi$. Then since $s(\psi)\in M_{\widetilde{\psi}}$ and $\sigma^{\widetilde{\psi}}_t ( \psi ) =  \psi$,  $u_t:=[D\varphi:D\widetilde{\psi}]_t$ satisfies
\begin{align*}
	\theta(\psi)
	= \Ad(u_t)\circ \sigma^{\widetilde{\psi}}_t (s(\psi)\widetilde{\psi})
	= u_t\psi u_t^*.
\end{align*}
Hence $\theta$ is pointwise inner.

By the uniqueness of Connes cocycles, it is straightforward to check that, if $\varphi,\psi$ are faithful,  $v \in M$ a partial isometry with $e:=v v^*\in M_\psi$, $v^*v \in M_\varphi$, and $v\varphi v^* = e\psi e$, then for any $x\in M$ and $t\in \R$,
	$$v\sigma_t^\varphi(v^*xv)v^*=\sigma_t^\psi(exe), \quad  e[D\psi : D\varphi]_t =  [De\psi e: D\varphi]_t=v\sigma_t^\varphi(v^*).$$
We prove two lemmas.

\begin{lem}\label{lem-RN-irreducible2}
	Let $N\subset M$ be an inclusion of $\sigma$-finite von Neumann algebras with expectation $E_N$. Let $\varphi_N\in N_\ast$ be any faithful tracial state and put $\varphi:=\varphi_N\circ E_N$. Let $p,q\in N$ be projections such that $N_p'\cap M_p\subset N_p$. Assume that there is a $\ast$-isomorphism $\theta\colon M_p \to M_q$ such that $\theta(N_p) = N_q$. We write $\theta_N:=\theta|_{N_p}$. 
\begin{enumerate}
	\item We have $\theta^{-1}\circ E_N = E_N\circ \theta^{-1}$ on $M_q$. In particular, 
	$$[Dq\theta(p\varphi p)q:D\varphi]_t = [Dq\theta_N(p\varphi_N p)q :D\varphi_N ]_t \in N,\quad \text{for all }t\in \R.$$

	\item There is a unique nosingular, positive, self-adjoint element $h$, which is affiliated with $N_q$, such that $q\theta(p\varphi p)q= \varphi_h $.

	\item Assume that $\theta = \Ad(v)\circ \theta'$ for some $\ast$-isomorphism $\theta'\colon pMp\to rMr$, where $r=\theta'(p)$, and for some partial isometry $v\in qMr$. Then $h$ in item $(2)$ satisfies
\begin{align*}
	 h^{\ri t} =  v[D r\theta'(p\varphi p)r : D\varphi]_t \sigma_t^{\varphi}(v^*),\quad \text{for all }t\in \R.
\end{align*}

\end{enumerate}
\end{lem}
\begin{proof}
	Recall that a normal conditional expectation from $M$ onto $N$ is unique if $N\subset M$ is irreducible. 
Observe that $N_p\subset M_p$ and $N_q = \theta(N_p)\subset \theta(M_p)=M_q$ are irreducible. 
In the proof below, for notation simplicity, we sometimes omit $p$, such as $q\theta(\varphi)q = q\theta(p\varphi p)q$.

	(1) Since $ \theta\circ E_N \circ \theta^{-1}\colon M_q \to N_q $ is a normal conditional expectation, it coincides with $E_N|_{M_q}$ by the uniqueness. We get $\theta_N^{-1}\circ E_N = E_N\circ \theta^{-1}$ on $M_q$, and hence 
	$$q\theta(\varphi)q = q( \varphi_N\circ E_N\circ\theta^{-1})q = q(\varphi_N \circ \theta_N^{-1} \circ E_N )q=q( \theta_N(\varphi_N) \circ E_N)q.$$
Take any faithful $\psi_N\in N_\ast^+$ with $q \psi_N q = q\theta(\varphi_N)q$. Then by \cite[Lemma 1.4.4]{Co72}, for any $t\in \R$,
 \begin{align*}
	[Dq\theta(\varphi)q : D\varphi]_t 
	&	= q [D\psi_N \circ E_N : D\varphi _N\circ E_N ]_t \\
	&=q [D\psi_N : D\varphi_N ]_t 
	= [Dq\theta_N(\varphi_N)q : D\varphi _N ]_t\in N.
\end{align*}

	(2) Since $\varphi_{N}$ is a faithful trace on $N$, there is a unique $h$, affiliated with $N_q$, such that $q\theta_N(\varphi_N)q = (\varphi_{N})_h$ on $N_q$. Combined with (1), for any $t\in \R$, we have for any $t\in \R$, 
	$$[Dq\theta(\varphi)q :D\varphi]_t=[Dq\theta_N(\varphi_N)q :D\varphi_N ]_t = h^{\ri t}.$$
Since $h$ is affiliated with $M_q$, we also have $h^{\ri t} = [D\varphi_h:D\varphi]_t$, hence we get $q\theta(\varphi)q =\varphi_h$.

	(3) Let $\psi\in M_\ast^+$ be any faithful element such that $r\psi r = r\theta'(\varphi)r$. Then for any $t\in \R$,
\begin{align*}
	h^{\ri t} 
	&= [Dq\theta(\varphi)q: D\varphi]_t 
	=  [D v\theta'(\varphi)v^*: D\varphi]_t\\
	&= [D v\psi v^*: D\psi]_t[D \psi: D\varphi]_t\\
	&= v\sigma_t^{\psi}(v^*)[D \psi: D\varphi]_t\\
	&= v[D \psi: D\varphi]_t \sigma_t^{\varphi}(v^*)
	= v[D r\psi r : D\varphi]_t \sigma_t^{\varphi}(v^*).
\end{align*}
This is the conclusion.
\end{proof}

\begin{lem}\label{lem-centralizer-singular}
	Let $M$ be a factor with a faithful state $\varphi \in M_\ast$ such that $M_\varphi' \cap M =\C$. If $\theta\in \Aut(M)$ satisfies $\theta(M_\varphi)=M_\varphi$, then $\theta(\varphi)=\varphi$. 
In particular, $M_\varphi \subset M$ is singular in the sense that $uM_\varphi u^* = M_\varphi$ for $u\in \mathcal U(M)$ implies $u\in M_\varphi$. 
\end{lem}
\begin{proof}
	Put $N:=M_\varphi$ and $\varphi_N:=\varphi|_N$. We apply Lemma \ref{lem-RN-irreducible2}(2) for the case $p=q=1$, and take $h$ such that $\theta(\varphi) = \varphi_h$. Since $h$ is affiliated with $N$, we have $\varphi_h|_N =(\varphi_N)_h$. Since $N$ is a finite factor and since $\theta(\varphi)|_N$ is a trace, we have $\theta(\varphi)|_N  = \varphi_N$ by the uniqueness of the trace. We get $(\varphi_N)_h = \varphi_N$ and $h=1$. This means $\theta(\varphi) = \varphi$. 

If  $uM_\varphi u^* = M_\varphi$ for $u\in \mathcal U(M)$, we can apply the first part of the lemma to $\Ad(u)\in \Aut(M)$. We get $\Ad(u)(\varphi)=\varphi$ and hence $u$ is contained in $M_\varphi$. 
\end{proof}

\subsection*{Ultraproduct von Neumann algebras}

	Let $M$ be any $\sigma$-finite von Neumann algebra and $\omega$ any free ultrafilter on $\N$. Put 
\begin{align*}
	\mathcal I_{\omega} &= \left\{ (x_n)_{n} \in \ell^\infty(\N, M) \mid x_n \to 0 \text{ $\ast$-strongly as } n \to \omega \right\} ;\\
	\mathcal M^{\omega} &= \left \{ x \in \ell^\infty(\N, M) \mid  x \mathcal I_{\omega} \subset \mathcal I_{\omega} \text{ and } \mathcal I_{\omega}x \subset \mathcal I_{\omega}\right\};
\end{align*}
where $\ell^\infty(\N, M)$ is the set of all norm bounded sequences in $M$. 
The \textit{ultraproduct von Neumann algebra} \cite{Oc85} is defined as the quotient C$^*$-algebra $M^\omega := \mathcal{M}^\omega/ \mathcal{I}_\omega$, which naturally admits a von Neumann algebra structure. The image of $(x_n)_n$ in $M^\omega$ is denoted by $(x_n)_\omega$. 
We have an embedding $M \subset M^\omega$ as a constant sequence, and this inclusion is with expectation $E_\omega$ given by $E_\omega( (x_n)_\omega)=\sigma\text{-weak}\lim_{n\to \omega}x_n$. For any faithful $\varphi\in M_*^+$, we can define a faithful positive functional $\varphi^\omega:=\varphi\circ E_\omega$ on $M^\omega$. 

Any $\theta\in \Aut(M)$ induces $\theta^\omega\in \Aut(M^\omega)$ by the equation $ \theta^\omega( (x_n)_\omega ) = (\theta(x_n))_\omega $. In particular, the modular action of $\varphi^\omega$ is given by $\sigma^{\varphi^\omega}_t = (\sigma^\varphi_t)^\omega$ for all $t\in \R$ \cite{AH12}. 
For more on ultraproduct von Neumann algebras, we refer the reader to \cite{Oc85,AH12}.

\subsection*{Relative bicentralizer algebras}

Let $N\subset M$ be an inclusion of von Neumann algebras with separable predual and with expectation $E_N$. Let $\varphi\in N_\ast$ be any faithful state and extend it on $M$ by $E_N$. 
	We define the \textit{asymptotic centralizers} and the \textit{relative bicentralizer algebra} (with respect to $\varphi$) as
\begin{align*}
	&\mathrm{AC}_\varphi(N):= \{ (x_n)_n\in \ell^\infty(\N,N) \mid \lim_{n\to \infty}\|x_n\varphi - \varphi x_n\| = 0  \};\\
	&\mathrm{BC}_\varphi(N\subset M):= \{x\in M  \mid \lim_{n\to \infty} \|x_n x - x x_n\|_\varphi \to 0 \quad \text{for any }(x_n)_n\in \mathrm{AC}_\varphi(N) \}.
\end{align*}
This does not depend on the choice of $\varphi$, if $N$ is a type $\rm III_1$ factor. In the case $N=M$, we write $\mathrm{BC}_\varphi(M)=\mathrm{BC}_\varphi(N\subset M)$ and we call it the \textit{bicentralizer}. Then Connes' bicentralizer problem asks if any type $\rm III_1$ factor $M$ with separable predual has trivial bicentralizer, that is, $\mathrm{BC_\varphi}(M)=\C$. This condition is equivalent to having a faithful state $\varphi\in M_\ast$ such that $M_\varphi'\cap M=\C$ \cite{Ha85}. Here is a convenient expression by ultraproducts \cite[Proposition 3.3]{HI15}\cite[Proposition 3.3]{AHHM18}:
	$$ \mathrm{BC}_\varphi(N\subset M) = (N^\omega)_{\varphi^\omega} '\cap M, $$
where $\omega$ is any fixed free ultrafilter on $\N$.

The next theorem explains the importance of relative bicentralizers. This should be understood as a type III counterpart of Popa's work \cite{Po81}.

\begin{thm}[{\cite{Ha85}\cite[Theorem C]{AHHM18}}]\label{thm-bicentralizer}
	Assume that $N$ is a type $\rm III_1$ factor. If $\mathrm{BC}_\varphi(N\subset M) = N'\cap M$, then there exists an amenable subfactor $R \subset N$ of type $\rm II_1$ with expectation such that $R'\cap M =N'\cap M$. 
\end{thm}

Very recently, Marrakchi obtained very general and useful criteria for computations of relative bicentralizers. We will need the following statement.

\begin{thm}[{\cite[Theorem E(5)]{Ma23}}]\label{thm-Amine}
	Assume that $N$ is a type $\rm III_1$ factor. Assume that $N$ has trivial bicentralizer and that $N\subset M$ is regular. If $N'\cap C_\varphi(M)\subset C_\varphi(N)$, then we have $\mathrm{BC}_\varphi(N\subset M)=\C$.
\end{thm}
\begin{proof}
	By \cite[Theorem E(5)]{Ma23}, we have
	$$\mathrm{BC}_\varphi(N\subset C_\varphi(M)) = L\R\vee (N'\cap C_\varphi(M))\subset C_\varphi(N),$$
where the relative bicentralizer algebra is defined for arbitrary inclusions. This implies
	$$\mathrm{BC}_\varphi(N\subset M) =M\cap \mathrm{BC}_\varphi(N\subset C_\varphi(M))\subset M\cap C_\varphi(N) = N.$$
We get $\mathrm{BC}_\varphi(N\subset M) \subset  \mathrm{BC}_\varphi(N) = \C$.
\end{proof}

\subsection*{Popa's intertwining theory}

We recall Popa's intertwining theory \cite{Po01,Po03}. 
For this, we fix a $\sigma$-finite von Neumann algebra $M$ and a finite von Neumann subalgebra $B\subset M$ with expectation $E_B$. Fix a trace $\tau_B$ on $B$. We can define a canonical trace $\Tr$ on the basic construction $\langle M, B\rangle$ satisfying $\Tr( xe_By) = \tau_B\circ E_B(yx)$ for $x,y\in M$ (see for example, \cite[Section 4]{HI15}). 

In this setting, we have the following equivalence. For the proof of this theorem, we refer the reader to \cite[Theorem 4.3]{HI15}. 

\begin{thm}[{\cite{Po01,Po03}}]\label{thm-intertwining}
	Keep the setting. For any finite von Neumann subalgebra $A\subset M$ with expectation, the following conditions are equivalent. 
	\begin{enumerate}
		\item We have $A \preceq_M B$. This means that there exist projections $e\in A$, $f\in B$, a partial isometry $v\in eMf$ and a unital normal $\ast$-homomorphism $\theta\colon eAe \to fBf$ such that $v\theta(a)= av$ for all $a\in eAe$. 

		\item There exists no nets $(u_i)_{i}$ of unitaries in $\mathcal U(A)$ such that for any $a,b\in M$,
	$$\|E_B(b^* u_i a )\|_{\tau_B} \rightarrow 0, \quad \text{as} \quad i\to \infty.$$

		\item There exists a nonzero positive element $d\in A' \cap \langle M, B\rangle$ such that $\Tr(d)\in M$. 

\end{enumerate}
\end{thm}

The next three lemmas are well known to experts. We include short proofs for the reader's convenience. 

\begin{lem}\label{lem-intertwiner-tensor}
	Keep the setting as in Theorem \ref{thm-intertwining}. Let $A_1,A_2\subset M$ be finite von Neumann subalgebras with expectation such that $A_1$ and $A_2$ are commuting with each other.

Suppose that there are nonzero positive elements $d_i\in A_i' \cap \langle M,B\rangle$ satisfying the condition of item $(3)$ in Theorem \ref{thm-intertwining}  for $i=1,2$ such that $d_1d_2\neq 0$. Then we have $A_1\vee A_2\preceq_M B$.
\end{lem}
\begin{proof}
	Consider $\Tr$ on $\langle M,B\rangle$ as in Theorem \ref{thm-intertwining}. Put
	$$ \mathcal K := \overline{\mathrm{conv}}^{\textrm{$\sigma$-weak}} \{ u d_1 u^* \mid u\in \mathcal U(A_2) \} \subset A_1'\cap \langle M,B\rangle.$$
By the normality of $\Tr$, any element in $\mathcal K$ has finite value in $\Tr$. Take the unique element $d\in \mathcal K$ which has the minimum distance from $0$. 
The uniqueness condition implies $udu^*=u$ for all $u\in \mathcal U(A_2)$, hence $d$ is contained in $(A_1\vee A_2)'\cap \langle M,B\rangle$. Then observe that $d_1,d_2\in L^2(\langle M,B\rangle,\Tr)$, and for any $u\in \mathcal U(A_2)$,
	$$\langle ud_1u^*,d_2 \rangle_{\Tr} = \Tr ( d_2 ud_1u^*) =  \Tr ( u^*d_2 ud_1)= \Tr(d_2d_1) = \Tr(d_1^{1/2} d_2 d_1^{1/2})>0.$$
This implies that $\mathcal K$ does not contain $0$, hence $d\neq 0$. This implies $A_1\vee A_2\preceq_M B$.
\end{proof}

\begin{lem}\label{lem-intertwining-core}
	Let $M$ be a von Neumann algebra, $\varphi\in M_\ast$ a faithful state, and $A,B\subset M_\varphi$ von Neumann subalgebras. If $A\not\preceq_M B$, then $A\ovt L\R \not\preceq_{C_\varphi(M)} B\ovt L\R$. 
\end{lem}
\begin{proof}
	Let $E_B\colon M\to B$ be the unique $\varphi$-preserving conditional expectation. Then since $E_B$ commutes with $\sigma^\varphi$, we can extend $E_B$ to one from $ C_\varphi(M)$ onto $B\ovt L\R$. Take any net $(u_i)_i$ of $\mathcal U(A)$ as in item (2) in Theorem \ref{thm-intertwining}. Then it is easy to see that $(u_i\otimes 1_{L\R})_i$ works for $A\ovt L\R\not\preceq_{C_\varphi(M)} B\ovt L\R$. 
\end{proof}

	In the lemma below, we need the case that $A,B \subset M$ are \textit{possibly non-unital} finite von Neumann subalgebras. In this case, we can use the same definition for $A\preceq_M B$, and item (2) in Theorem \ref{thm-intertwining} works by replacing $a,b\in M$ by $a,b\in M1_B$, where $1_B$ is the unit of $B$. Item (3) is slightly more complicated, but we do not need it.  See \cite[Theorem 4.3]{HI15} for more on the non-unital case.

\begin{lem}\label{lem-intertwining-finitelymany}
	Let $A,B_1,\ldots,B_n\subset M$ be (possibly non-unital) inclusions of $\sigma$-finite von Neumann algebras with expectation. Let $E_{B_k}\colon M_{1_{B_k}}\to B_k$ be faithful normal conditional expectations for all $k$ and assume that $A,B_1,\ldots,B_n$ are finite with trace $\tau_k\in (B_k)_\ast$ for all $k$. If $A\not\preceq_M B_k$ for all $k$, then there exists a net $(u_i)_{i}$ in $\mathcal U(A)$ such that for all $k$,
	$$\|E_{B_k}(b^* u_i a ) \|_{\tau_k}\rightarrow 0, \quad \text{for all $a,b\in M1_{B_k}$.}$$
\end{lem}
\begin{proof}
	Consider $M_n:=M\otimes \C^n$ and embeddings 
	$$\widetilde{A}:=A\otimes \C1_{\C^n}\subset M_n,\quad \widetilde{B}:=B_1\oplus\cdots\oplus B_n \subset M_n.$$
Observe that $A\not\preceq_M B_k$ implies $\widetilde{A}\not\preceq_{M_n} B_k$ for all $k$ (use item (2) of Theorem \ref{thm-intertwining}).  Then by \cite[Remark 4.2(2)]{HI15}, it holds that $\widetilde{A}\not\preceq_{M_n} \widetilde{B}$. Then any net of unitaries in item $(2)$ of Theorem \ref{thm-intertwining} for $\widetilde{A}\not\preceq_{M_n} \widetilde{B}$ works.
\end{proof}

\section{Proof of Theorem \ref{thmA}: (1)$\Rightarrow$(2)}

We need two lemmas. The first one uses the \textit{measurable selection principle}.

\begin{lem}\label{lem-section}
	Let $M$ be a von Neumann algebra with separable predual, $\varphi$ a faithful normal state, and  $\theta \in \Aut(M)$ a pointwise inner automorphism such that $\theta(\varphi)=\varphi$. 
Then for any positive invertible element $a\in M_\varphi$, there exist a compact subset $K\subset (0,1)$ with positive Lebesgure measure, and a continuous map $K \ni p\mapsto u_p\in \mathcal U(M)$ such that
	$$ \theta(\varphi_{a^p}) = u_p \varphi_{a^p} u^*_p ,\quad \text{for all}\quad p\in K,$$
where $\mathcal U(M)$ is equipped with the $\ast$-strong topology.
\end{lem}
\begin{proof}
	Throughout the proof, we consider the $\ast$-strong topology for $\mathcal U(M)$ and we regard it as a Polish space. Set
	$$\mathcal S:= \{ (p,u)\in (0,1)\times \mathcal U(M)\mid \theta(\varphi_{a^p}) = u \varphi_{a^p} u^* \} .$$
Then it is easy to see that $\mathcal S$ is a closed subset. Since $\theta$ is pointwise inner, we have $(0,1)=\pi_1(\mathcal S)$, where $\pi_1$ is the projection onto the first coordinate. Hence, with the Lebesgue measure on $(0,1)$, we can apply the measurable selection principle (e.g.\ \cite[Theorem 14.3.6]{KR97}) and find a Lebesgue measurable map 
	$$\eta \colon (0,1)=\pi_1(\mathcal S) \to \mathcal U(M)$$
such that $(p,\eta(p)) \in \mathcal S$ for all $p\in (0,1)$. Since $\mathcal U(M)$ is a Polish space, we can apply Lusin's theorem to $\eta$ and find a compact subset $K\subset (0,1)$ such that $\eta|_{K}$ is continuous and that $(0,1)\setminus K$ has arbitrary small Lebesgue measure. We get the conclusion.
\end{proof}

The second lemma uses the first one. This lemma is the key observation of the article.

\begin{lem}\label{key-lemma}
	Let $M$ be a von Neumann algebra with separable predual, $\varphi\in M_\ast$ a faithful state, $\theta\in \Aut(M)$ a pointwise inner automorphism such that $\theta(\varphi)=\varphi$. Then for any positive invertible element $a\in M_\varphi$, we have $ A\preceq_M \theta(A)  $, where  $A:=\mathrm{W}^*\{a\}\subset M_\varphi$. 
\end{lem}
\begin{proof}
	Observe that $\theta(M_\varphi)=M_\varphi$ and $\theta(A)\subset M_\varphi$. Take any $u\in \mathcal U(M)$ such that $\varphi_{\theta(a)} = \theta(\varphi_a) = u \varphi_a u^*$. Then since
\begin{align*}
	&[D\varphi_{\theta(a)} :D\varphi_a]_t 
	= [D\varphi_{\theta(a)} :D\varphi]_t[D\varphi :D\varphi_a]_t
	= \theta(a^{\ri t})a^{-\ri t}, \quad \text{and}\\
	& [Du \varphi_a u^* :D\varphi_a]_t 
	= u\sigma_t^{\varphi_a} (u^*)
	= u a^{\ri t} \sigma_t^{\varphi} (u^*) a^{-\ri t} ,
\end{align*}
we get $ \theta(a^{\ri t}) = u a^{\ri t} \sigma_t^{\varphi} (u^*)$ for all $ t\in \R $. 
Define
	$$ \widetilde{A}:=\mathrm{W^*}\{ a^{\ri t}\otimes \lambda_t \mid t\in \R \}\subset C_\varphi(M), $$
where $\lambda_t$ denotes the left regular representation on $L^2(\R)$, and observe that $\widetilde{A}\subset C_\varphi(M)$ is with expectation. Then it is easy to see that $\Ad(u)$ restricts to a $\ast$-homomorphism from $\widetilde{A}$ into $\theta(A)\ovt L\R$, so that we get $\widetilde{A}\preceq_{C_\varphi(M)} \theta(A)\ovt L\R$. 
More precisely, if we denote by $e$ the Jones projection for any fixed faithful normal conditional expectation from $C_\varphi(M)$ onto $\theta(A)\ovt L\R$, then the projection $u^* e u$ satisfies the condition of item $(3)$ in Theorem \ref{thm-intertwining}.

	Fix $0<p<1$ and we consider the positive invertible element $a^p$. Then by the same reasoning as above, it holds that 
	$$ \widetilde{A}_p\preceq_{C_\varphi(M)} \theta(A)\ovt L\R, \quad \text{where}\quad \widetilde{A}_p:=\mathrm{W^*}\{ a^{\ri pt}\otimes \lambda_t \mid t\in \R \},$$
together with the projection $u_p^* e u_p$, where $u_p$ is any unitary satisfying $\varphi_{\theta(a^p)} =  u_p \varphi_{a^p} u_p^*$. We note that the algebra $\theta(A)\ovt L\R$ does not depend on $p$.

	Now we claim $A\ovt L\R \preceq_{C_\varphi(M)} \theta(A)\ovt L\R$. To see this, we first apply Lemma \ref{lem-section} and find a compact $K\subset (0,1)$ with positive measure and a continuous map $K\ni p \mapsto u_p$. Then take any sequence $p_n\in K$ such that $p_n $ converges to $p\in K$ and $p\not\in \{p_n\}_n$. Then $u_{p_n} \to u_p$ converges in the $\ast$-strong topology, hence 
	$$d_{p_n}:=u_{p_n}^* e u_{p_n}\to u_{p}^* e u_{p}=:d_p $$
as well. In particular, we can find $p_n=:q$ such that $q\neq p$ and $d_q d_p\neq 0$. Then by Lemma \ref{lem-intertwiner-tensor}, it holds that 
	$$\widetilde{ A}_p\vee \widetilde{A}_q \preceq_{C_\varphi(M)} \theta(A)\ovt L\R. $$
Since $p\neq q$, the left hand side coincides with $A\ovt L\R$ and the claim is proven. 

Finally, this condition implies $A\preceq_M \theta(A)$ by Lemma \ref{lem-intertwining-core}.
\end{proof}

Now we prove (1)$\Rightarrow$(2) of Theorem \ref{thmA}. For the proof, we need Popa's $\mathcal G$-singular masa technique for type $\rm III_1$ factors, see Theorem \ref{thm-typeIII-G-singular}.

\begin{proof}[{\bf Proof of Theorem \ref{thmA}:(1)$\Rightarrow$(2)}]
	Let $\psi\in M_\ast^+$ be any faithful state such that $M_\psi'\cap M=\C$. Since $\theta$ is pointwise inner, there is $v\in \mathcal U(M)$ such that $\Ad(v)\circ \theta(\psi)=\psi$. So up to replacing $\Ad(v)\circ \theta$ by $\theta$, we may assume $\theta(\psi)=\psi$. Then we show that $\theta|_{M_\psi}$ is inner. 

Suppose by contradiction that $\theta$ is not inner on $M_\psi$. Then putting $\mathcal G:=\{\id_M,\ \theta^{-1}\}$, we apply Theorem \ref{thm-typeIII-G-singular} and take a $\mathcal G$-singular masa $A\subset M_\psi$. 
Take any generator $a$ of $A$ which is positive and invertible. We can apply Lemma \ref{key-lemma} to $a$ and get $A\preceq_M \theta(A)$. Since $A$ and $\theta(A)$ are masas in $M$, one can find a partial isometry $v\in M$ such that 
	$$ v^*v \in A,\quad vv^* \in \theta(A),\quad v A v^* \subset \theta(A)vv^* .$$
(For this, see for example the proof of \cite[Lemma F.17]{BO08}.) Thus $\Ad(v)$ restricts to a $\ast$-isomorphism from $Av^*v$ onto $\theta(A)vv^*$. 
Since $A$ is diffuse, up to exchanging $v^*v$ with a smaller projection in $A$ (or $vv^*$ with one in $\theta(A)$) if necessary, we may assume $v^*v\neq 1$ and $vv^*\neq 1$. 
Then since $M$ is a type III factor, $1-v^*v$ and $1-vv^*$ are equivalent, so we can take a partial isometry $w\in M$ such that $w^*w=1-v^*v$ and $ww^* = 1-vv^*$. 
Then $\widetilde{v}:= v+w\in M$ is a unitary element satisfying $\widetilde{v}p = v$ where $p:=v^*v$, so that 
	$$ \widetilde{v} A p \widetilde{v}^* = v A v^*\subset \theta(A).$$
This means $ \Ad(\theta^{-1}(\widetilde{v})) \circ \theta^{-1} (A p)  \subset A.$
By the $\mathcal G$-singularity, it holds that $\theta^{-1}=\id_M$, a contradiction.
\end{proof}

\section{Proof of Theorem \ref{thmA}: (2)$\Rightarrow$(3) and (3)$\Rightarrow$(4)}

\begin{proof}[{\bf Proof Theorem \ref{thmA}: (2)$\Rightarrow$(3)}]
	We may assume that $\theta$ is outer on $M$. Set $M_2:=M\otimes \M_2$ and consider an embedding $M\subset M_2$ given by 
	$$  M \ni x \mapsto \left[\begin{array}{cc}
x & 0 \\
0 & \theta(x)
\end{array}\right] \in M_2.$$
Since $\theta$ is outer, we have $M'\cap M_2 = \C \oplus \C$. We consider the relative bicentralizer  
	$$ \mathrm{BC}_\varphi(M\subset M_2) = (M^\omega)_{\varphi^\omega}' \cap M_2. $$
We claim $\mathrm{BC}_\varphi(M\subset M_2)\neq M'\cap M_2$. 

To see this, suppose by contradiction that $\mathrm{BC}_\varphi(M\subset M_2)=M'\cap M_2$. Then by Theorem \ref{thm-bicentralizer}, there is an amenable $\rm II_1$ factor $R\subset M$ with expectation $E_R$ such that $R'\cap M_2 = M'\cap M_2 = \C\oplus \C$. 
Put $\psi:=\tau_R\circ E_R$, where $\tau_R$ is the trace on $R$, and observe that $R\subset M_\psi \subset M$. Then since $M_\psi'\cap M\subset R'\cap M=\C$ (this follows by $R'\cap \widetilde{M}=\C\oplus \C$), by assumption, there is $u\in \mathcal U(M)$ such that $\Ad(u)\circ \theta|_{M_\psi} = \id_{M_\psi}$. 
Then  $\left[\begin{array}{cc}
0 & u  \\
0 & 0
\end{array}\right]$ is contained in $M_\psi'\cap M_2 \subset R'\cap M_2 = \C\oplus \C$. This is a contradiction and the claim is proven.

We take any $\left[\begin{array}{cc}
a & b  \\
c & d
\end{array}\right]\in (M^\omega)_{\varphi^\omega}' \cap M_2$, which is not contained in $\C\oplus \C$. Then since $(M^\omega)_{\varphi^\omega}' \cap M = \mathrm{BC}_\varphi(M)=\C$, we have $a,d\in \C$, hence $c\neq 0$ or $d\neq 0$. By taking adjoint if necessary, we may assume $c\neq 0$. Then observe that $\theta(x)c = cx$ for all $x\in (M^\omega)_{\varphi^\omega}$. Since $(M^\omega)_{\varphi^\omega}' \cap M =\C$, by the polar decomposition, we can replace $c$ by a unitary element. This means (3) holds. 
\end{proof}

\begin{proof}[{\bf Proof Theorem \ref{thmA}: (3)$\Rightarrow$(4)}]
By assumption, take $u\in \mathcal U(M)$ such that $\Ad(u)\circ \theta = \id$ on $(M^\omega)_{\varphi^\omega}$. Up to replacing $\Ad(u)\circ \theta$ by $\theta$, we may assume $u=1$. Then by Lemma \ref{lem-centralizer-singular}, it holds that $\theta(\varphi)=\varphi$. In particular, $\theta$ commutes with the modular action of $\varphi$.

We set
	$$H(\theta):= \{ n\in \Z\mid \theta^n = \Ad(u)\circ \sigma^\varphi_t \quad \text{for some }u\in \mathcal U(M), \ t\in \R\}.$$
Observe that by $M_\varphi'\cap M=\C$ and $\theta|_{M_\varphi}=\id$, the equality $\theta^n = \Ad(u)\circ \sigma^\varphi_t$ implies $u\in \C$, so we can actually remove $u\in \mathcal U(M)$ in the definition of $H(\theta)$. 
Since $H(\theta)$ is a subgroup in $\Z$, we can write $H(\theta)= k\Z$ for some $k\geq 0$. 

Assume that $k>0$ and $\theta^k= \sigma^\varphi_t $. Then putting $\widetilde{\theta}:=\sigma^\varphi_{-t/k}\circ \theta$, observe that $H(\widetilde{\theta})=H(\theta)$, $\widetilde{\theta}^k=\id$, and $\widetilde{\theta}=\id$ on $(M^\omega)_{\varphi^\omega}$. Hence up to replacing $\widetilde{\theta}$ by $\theta$, we may assume $\theta^k =\id$. If $k=0$, we do not need this replacement.

For any $k\geq 0$, set $G:=\Z/k\Z$ and define an outer action $G\actson M$ given by $\{\theta^n\}_{n\in \Z}$. We denote this action again by $\theta$. Put $\widetilde{M}:=M\rtimes_\theta G$ with canonical expectation $E_M\colon \widetilde{M}\to M$ and note that $M'\cap \widetilde{M}=\C$ by the outerness of $\theta$. We extend $\varphi$ on $\widetilde{M}$ by $E_M$. Since $\theta$ commutes with the modular action of $\varphi$, the continuous core of $\widetilde{M}$ is of the form that 
	$$C_\varphi(\widetilde{M}) = M\rtimes_{\theta\times \sigma^\varphi} (G\times \R) = C_\varphi(M)\rtimes_\theta G,$$
where $\theta\colon G\actson C_\varphi(M)$ in the right hand side is the canonical extension. 

The next claim is important to us. 

\begin{claim}
	We have $M'\cap C_\varphi(\widetilde{M}) = \C$.
\end{claim}
\begin{proof}[{\bf Proof of Claim}]
	Since the canonical extension $\theta \colon G\actson C_\varphi(M)$ is still outer (e.g.\ \cite[Lemma XII.6.14]{Ta03}), the crossed product $C_\varphi(M)\rtimes_\theta G$ is a factor. 
Let $E:= \varphi\otimes \id\colon \B(L^2(M)) \ovt \B(L^2(G\times \R)) \to \B(L^2(G\times \R))$ and observe that it restricts to a conditional expectation $M\rtimes(G\times \R) \to L(G\times \R)$ (because $\theta$ and $\sigma^\varphi$ preserves $\varphi$). 

Since the $G\times \R$ action is trivial on $M_{\varphi}$, it holds that
\begin{align*}
	M_{\varphi}'\cap C_\varphi(\widetilde{M})
	\subset (M_{\varphi}' \cap M)\ovt \B(L^2(G\times \R)) = \C \ovt \B(L^2(G\times \R)) .
\end{align*}
We have $E(x)=x$ for all $x\in M_{\varphi}'\cap C_\varphi(\widetilde{M})$, so that $M_{\varphi}'\cap C_\varphi(\widetilde{M})=L(G\times \R)$. Then since $G\times \R$ is abelian, 
	$$ M'\cap C_\varphi(\widetilde{M}) = M'\cap L(G\times \R) \subset \mathcal Z(C_\varphi(\widetilde{M})).$$
Since $C_\varphi(\widetilde{M})$ is a factor, we get the conclusion.
\end{proof}

To finish the proof, we have to show $k=1$. For this, suppose by contradiction that $k\neq 1$ and $G=\Z/k\Z$ is a nontrivial group. Then observe that the inclusion $M\subset \widetilde{M}$ is with expectation, irreducible, regular, $\mathrm{BC}_\varphi(M)=\C$, and $M'\cap C_\varphi(\widetilde{M})\subset C_\varphi(M)$. Hence we can apply Theorem \ref{thm-Amine} and get that $\mathrm{BC}_\varphi(M\subset \widetilde{M}) =\C$. 
Since $G$ is nontrivial, $1\in G$ (the image of $1\in \Z$ in $G$) is nontrivial. We denote by $\lambda_1^G\in LG$ the canonical unitary corresponding to $1\in G$. Then since $\theta=\id$ on $(M^\omega)_{\varphi^\omega}$ and since $\Ad(\lambda_1^G)\in \Aut(\widetilde{M})$ coincides with $\theta$ on $M$, we get 
	$$ \lambda_1^G \in (M^\omega)_{\varphi^\omega}' \cap \widetilde{M} = B_\varphi(M\subset \widetilde{M}) =\C. $$
This is a contradiction because $1\in G$ is nontrivial. Thus we get $k=1$ and this finishes the proof.
\end{proof}

\appendix 
\section{Popa's $\mathcal G$-singular masas in type $\rm III_1$ factors}

The following theorem is an analogue of \cite[Theorem 4.2]{Po83} in the type III setting. It covers the case that $M$ is a type $\rm II_1$ factor, a type $\rm III_\lambda$ factor $(0<\lambda<1)$, or a type $\rm III_1$ factor with trivial bicentralizer. The proof is a rather straightforward adaptation of \cite{Po16,HP17}. So we will give a sketch of the proof.

\begin{thm}\label{thm-typeIII-G-singular}
	Let $M$ be a factor with separable predual and $\varphi\in M_\ast$ a faithful state such that $M_\varphi'\cap M=\C$. Let $\{\theta_n\}_{n\geq 1}\subset \Aut(M)$ be a countable subset such that 
	$$ \theta_n(\varphi)=\varphi\quad \text{and}\quad \theta_n|_{M_\varphi}\not\in \Int(M_\varphi) $$
for all $n\geq 1$. Put $\theta_0:=\id_M$ and $\mathcal G:=\{\theta_n\}_{n\geq 0}$. 

Then there is an abelian von Neumann subalgebra $A\subset M_\varphi$ such that $A$ is a masa in $M$ and that $A$ is $\mathcal G$-singular in the following sense: if $\theta\in \mathcal G$ satisfies 
	$$\Ad(u)\circ \theta(Ap)\subset A$$
for some nonzero projection $p\in A$ and $u\in \mathcal U(M)$, then $\theta=\id_M$ and $up=pu\in A$.
\end{thm}

We note that, since $\id_M$ is contained in $\mathcal G$, the resulting masa $A\subset M$ is \textit{singular} in the sense that $uAu^*=A$ for $u\in \mathcal U(M)$ implies $u\in A$.

To use ideas of \cite{Po16} in our setting, the next lemma is important.

\begin{lem}\label{lem-centralizer-outer}
	Keep the setting as in Theorem \ref{thm-typeIII-G-singular}. For any $\theta\in \Aut(M)$, define an embedding
	$$ \iota_{\theta}\colon M \ni x \mapsto \left[\begin{array}{cc}
x & 0 \\
0 & \theta (x)
\end{array}\right] \in M\otimes \M_2.$$
If $\theta(\varphi)=\varphi$, then the following conditions are equivalent:
\begin{enumerate}

	\item $\theta|_{M_\varphi}$ is outer;

	\item $\iota_\theta(M_\varphi)'\cap (M\otimes \M_2 ) \subset M\oplus M$.

\end{enumerate}
\end{lem}
\begin{proof}
	If $\theta|_{M_\varphi} = \Ad(u)$ for some $u\in \mathcal U(M_\varphi)$, then $\left[\begin{array}{cc}
0 & 0 \\
u & 0
\end{array}\right] $ is contained in $\iota_\theta(M_\varphi)'\cap M\otimes \M_2 $. 

If (2) does not holds, then there is a nonzero $a\in M$ such that $\theta(x)a=ax$ for all $x\in M_\varphi$. By the polar decomposition and by $M_\varphi'\cap M=\C$, there is a unitary $u\in \mathcal U(M)$ such that $\theta(x)=uxu^*$ for all $x\in M_\varphi$. Then $\Ad(u)$ globally preserves $M_\varphi$, hence $u$ is in $M_\varphi$ by Lemma \ref{lem-centralizer-singular}. Thus (1) does not hold.
\end{proof}

\begin{proof}[{\bf Proof of Theorem \ref{thm-typeIII-G-singular}}]
	We fix the following setting.
\begin{itemize}
	\item Let $M$ be a von Neumann algebra with separable predual, $\varphi\in M_\ast$ a faithful state, and $N\subset M_\varphi$ a type $\rm II$ von Neumann subalgebra such that $N'\cap M\subset N$. Set 
	$$ \Aut_\varphi(N\subset M):= \{ \theta\in \Aut(M)\mid \theta(\varphi)=\varphi,\ \theta(N)=N \} .$$

	\item Let $\{\theta_n\}_{n\geq 1}\subset \Aut_\varphi(N\subset M)$ be such that
	$$\iota_{n}(N)'\cap (M\otimes \M_2)\subset M\oplus M, \quad \text{for all }n\geq 1,$$
where $\iota_n:=\iota_{\theta_n}$ is as in Lemma \ref{lem-centralizer-outer}.
\end{itemize}
	We will find $A\subset N$ that satisfies the conclusion of Theorem \ref{thm-typeIII-G-singular} for $\mathcal G=\{\theta_n\}_{n\geq 0}$ with $\theta_0=\id$. Thanks to Lemma \ref{lem-centralizer-outer}, this gives the proof of Theorem \ref{thm-typeIII-G-singular}. 

For any von Neumann subalgebra $B\subset M$ that is globally invariant by $\sigma^\varphi$, we denote by $E_B\colon M\to B$ the unique $\varphi$-preserving conditional expectation, and by $e_N$ the corresponding Jones projection.

\begin{lem}\label{lem-G-singular-masa-construction}
	Keep the setting. We fix a finite dimensional von Neumann subalgebra  $A\subset N$, a projection $f\in A'\cap N$, and finite subsets $X\subset M$ and $Y\subset \Aut_\varphi(N\subset M)$. 

Then for any $\varepsilon>0$, there exist a finite dimensional von Neumann subalgebra $D\subset N$, which contains $A$ and $f$, and $v \in \mathcal U(Df)$ such that
	$$  \| E_{D'\cap M}( y^*\theta(v) x  )f^\perp\|_{\varphi}< \varepsilon\quad \text{for all }x,y\in X \text{ and  } \theta\in Y.$$
\end{lem}
\begin{proof}
	Since $A\vee \C f \subset N$ is finite dimensional, by \cite[Theorem 3.2]{Po81} (see also \cite[Lemma 2.2]{HP17}), there is a commutative von Neumann subalgebra $B\subset N$ that is a masa in $M$ and that contains $A\vee \C f$. 
We can write $B=\bigvee_k B_k$, where $\{B_k\}_k$ is an increasing sequence of finite dimensional von Neumann algebras such that $B_0:=A\vee \C f$.

Since $A_f'\cap N_f$ is of type $\rm II$ and $B$ is of type I, we have
	$$ A_ f'\cap N_f\not \preceq_M \theta^{-1}(B _{f^\perp}),\quad \text{for all}\quad \theta\in Y.$$
By Lemma \ref{lem-intertwining-finitelymany}, take $v\in \mathcal U(A_f'\cap N_f)$ such that for all $\theta\in Y$,
	$$ \|  E_{B}(f^\perp y^* \theta(v) xf^\perp)\|_\varphi=\| E_{\theta^{-1}(Bf^\perp)}(  \theta^{-1}(f^\perp y^*)  v  \theta^{-1}(xf^\perp ) )\|_\varphi<\varepsilon,\quad \text{for all}\quad x,y\in X.$$
By approximating $v$ by finite sums of projections, we may assume that $v$ is contained in a finite dimensional abelian subalgebra in $A_f'\cap N_f$. Hence there exists a finite dimensional commutative von Neumann algebra $D_1 \subset N_f$ containing $v$ and $A_f$. 

Since $ e_{B_k'\cap M } \to e_{B'\cap M} = e_{B}$ converges strongly, there exists $k$ such that
	$$  \| E_{B_k'\cap M } ( y^* \theta(v) x  )f^\perp\|_\varphi=\| E_{B_k'\cap M } (f^\perp y^* \theta(v) x f^\perp )\|_\varphi <\varepsilon,\quad \text{for all}\quad x,y\in X,\quad \theta\in Y.$$
Here we can exchange $E_{B_k'\cap M }$ by $E_{D_2'\cap M_{f^\perp} } $, where  $D_2:=(B_k)_{f^\perp}$. Then $D:=D_1\oplus D_2 \subset N_f \oplus N_{f^\perp}\subset N$ works. 
\end{proof}

	Take any $\ast$-strongly dense subset $\{x_n\}_{n\in \N}\subset (M)_1$. Take any projections $\{e_n\}_{n\in \N}$ in $N$ which is $\ast$-strongly dense in all projections in $N$. We may assume that $e_0=x_0=1$ and that each $e_k$ appears infinitely many times in $\{e_n\}_{n\in \N}$. 

For each $n>0$, we denote the inclusion $\iota_n(M)\subset M\otimes \M_2$ by $M\subset M_n$. For $n=0$, we put $M_0:=M$.

\begin{claim}
	There exist an increasing sequence of finite dimensional von Neumann subalgebras $A_n\subset N$, projections $f_n\in A_n$, unitaries $v_n \in \mathcal U(A_nf_n)$ such that
\begin{itemize}
	\item[$\rm (P1)$] $\|f_n - e_n\|_\varphi\leq 7 \| e_n-E_{A_{n-1}'\cap N}(e_n) \|_\varphi$;

	\item[$\rm (P2)$] $ \| E_{A_{n}'\cap M}( x_i^* \theta_k(v_n) x_j  )f_n^\perp\|_{\varphi}\leq \frac{1}{n}$\quad for all $0\leq i,j,k\leq n$;

	\item[$\rm (P3)$] $ \| E_{A_{n}'\cap M_k}(x_j) - E_{A_n\vee (N'\cap M_k)}(x_j) \|_{\varphi}\leq \frac{1}{n}$\quad for all $0\leq j,k\leq n$.

\end{itemize}
\end{claim}
\begin{proof}
	We can follow the proof of the first part of \cite[Theorem 1]{HP17}, by using Lemma \ref{lem-G-singular-masa-construction} and \cite[Lemma 1.2.2]{Po16} for (P2) and (P3) respectively.
\end{proof}

Using $A_n$ in the claim, we define $A:=\bigvee_n A_n \subset N$. Then condition (P3) implies 
	$$A'\cap M_n=   A\vee (N'\cap M_n)\quad \text{for all}\quad n\geq 0.$$
Since $N'\cap M_n \subset M\oplus M$ for $n\geq1$ by assumption, this means that $A$ is a masa in $M$ and that $A'\cap M_n = A\oplus A$ for $n\geq 1$. 

Suppose now that there exist $\theta\in \mathcal G$, $u\in \mathcal U(M)$, a projection $p\in A$ such that
	$$\theta^u(Ap) \subset A ,\quad \text{where}\quad \theta^u:=\Ad(u)\circ \theta.$$
Since $A,\theta^u(A)$ are masas in $M$, putting $q:=\theta^u(p)\in A$, we have a $\ast$-isomorphism
	$$\theta^u \colon pMp \to qMq\quad \text{such that}\quad \theta^u(Ap)=Aq.$$
Assume first $\theta=\id$ and $\theta^u|_{Ap}=\id_{Ap}$. Then since $A$ is a masa in $M$, we get $up=pu\in A$ that is the conclusion. 
So from now on, we assume either that $\theta\neq \id$, or $\theta=\id$ and $\theta^u|_{Ap}\neq \id_{Ap}$. Then we will deduce a contradiction.

\begin{claim}
	There exists a nonzero projection $z\in Ap$ such that $\theta^u(z)z=0$.
\end{claim}
\begin{proof}
	If $\theta = \theta_n$ $(n\geq 1)$ and $\theta^u|_{Ap} = \id_{Ap}$, then $ \left[\begin{array}{cc}
0 & pu \\
0 & 0
\end{array}\right] $ is contained in $A'\cap M_n = A\oplus A$, hence a contradiction. So we have $\theta^u|_{Ap} \neq \id_{Ap}$ in this case. The same holds if $\theta=\theta_0=\id_M$ by assumption. 

Put $q:=\theta^u(p)$. If $p=q$, then $\theta^u$ defines a nontrivial automorphism on $Ap$, so that $z\in Ap$ exists. If $p\neq q$, then putting $z_0:=p-pq\in Ap$ and $w_0:=q-pq\in Aq$, we have $z_0\neq 0$ or $w_0\neq 0$. If $z_0\neq 0$, then $z:=z_0$ works. If $w_0\neq 0$, then $z:=(\theta^u)^{-1}(w_0)$ works.
\end{proof}

We apply Lemma \ref{lem-RN-irreducible2}(3) to $A\subset M$ and $(\theta^u)^{-1} \colon pMp \to qMq$ with $(\theta^u)^{-1}(Ap)=Aq$. By using $(\theta^u)^{-1}=\Ad(\theta^{-1}(u^*))\circ \theta^{-1}$ and $\theta(\varphi)=\varphi$, there is a unique nonsingular positive element $h$, which is affiliated with $Ap$, such that
\begin{align*}
	p(\varphi\circ \theta^u)p = \varphi_h,\quad \sigma_t^{\varphi}(u\theta(p))=u\theta(h^{\ri t}) ,\quad t\in \R.
\end{align*}
There exists a spectral projection $p_0\in Ap$ of $h$ such that
	$$ p_0z\neq 0,\quad \delta p_0 \leq hp_0 \leq \delta^{-1}p_0\quad \text{for some }\delta>0 .$$
To deduce a contradiction, up to replacing $p_0,zp_0,\theta^u(p_0)$ by $p,z,q$, we may assume $p=p_0$ and $z\leq p$. Then $u\theta(z)$ is a partial isometry, for which there is $C>0$ such that
	$$ \|x u\theta(z)\|_\varphi \leq \kappa_1 \|x\|_\varphi,\quad  \|x (u\theta(z))^*\|_\varphi \leq \kappa_2 \|x\|_\varphi ,\quad \text{for all } x\in M.$$
Now we follow the proof of \cite[Theorem 1]{HP17}, and get that
\begin{enumerate}
	\item[$\rm (a)$] $\|z f_n\|_\varphi = \|\theta(v_nz)\|_\varphi\leq \kappa_1 \|E_{A_n'\cap M}(u \theta(v_nz) u^*)z^\perp \|_\varphi $, \quad for all $n\in \N$;

	\item[$\rm (b)$] $\|E_{A_n'\cap M}(u \theta(v_n z)u^*)z^\perp \|_\varphi 
	\leq \|E_{A_n'\cap M}( y^* \theta(v_n) x)z^\perp \|_\varphi + \kappa_2 \| u-y^* \|_\varphi + \|  zu^* - x\|_\varphi $, \quad for all $x,y\in M$;

	\item[$\rm (c)$] $\|e_n - f_n\|_\varphi \leq 7\|e_n-z\|_\varphi$, \quad for all $n\in \N$.

\end{enumerate}
Take a subsequence $\{n_k\}_k$ such that $e_{n_k} \to z$ strongly. Then item (c) implies $e_{n_k} - f_{n_k}$ converges to $0$ strongly, so that $zf_{n_k}\to z$ as well. We show that $zf_{n_k}$ also converges to $0$, a contradiction. 

Fix any $\varepsilon>0$ and take $x_i,x_j$ such that $\kappa_2\| u-x_i^* \|_\varphi + \|  zu^* - x_j\|_\varphi <\varepsilon$. Then by items (a) and (b), 
\begin{align*}
	\|z f_n\|_\varphi
	&\leq \|E_{A_n'\cap M}(u \theta(v_n z) u^*)z^\perp \|_\varphi \\
	&\leq \|E_{A_n'\cap M}( x_i^* \theta(v_n) x_j)z^\perp \|_\varphi + \varepsilon\\
	&\leq \|E_{A_n'\cap M}( x_i^* \theta(v_n) x_j)f_n^\perp \|_\varphi +\|z^\perp - f_n^\perp\|_\varphi+ \varepsilon.
\end{align*}
Then (P2) implies
\begin{align*}
	\limsup_k\|z f_{n_k}\|_\varphi
	&\leq \limsup_k(\|E_{A_{n_k}'\cap M}( x_i^* \theta(v_{n_k}) x_j)f_{n_k}^\perp \|_\varphi +\|z^\perp - f_{n_k}^\perp\|_\varphi)+ \varepsilon = \varepsilon.
\end{align*}
Letting $\varepsilon\to0$, we get $zf_{n_k}\to 0$ strongly, as desired.
\end{proof}

\end{document}